\documentclass[wcp]{jmlr}

\usepackage{makecell}
\usepackage{lscape}
\usepackage{adjustbox}
\newcommand{\tabitem}{~~\llap{\textbullet}~~}

\usepackage{longtable}

\usepackage{microtype}
\usepackage{booktabs}

\usepackage{preamble}
\makeatletter
\def\set@curr@file#1{\def\@curr@file{#1}} 
\makeatother
\usepackage[load-configurations=version-1]{siunitx} 

\jmlrvolume{129}
\jmlryear{2020}
\jmlrworkshop{ACML 2020}

\title[Convergence Rates of a Momentum Algorithm with Bounded Adaptive Step Size]
{Convergence Rates of a Momentum Algorithm with Bounded Adaptive Step Size for Nonconvex Optimization}

 \author{\Name{Anas Barakat} \Email{anas.barakat@telecom-paris.fr}\\
  \Name{Pascal Bianchi} \Email{pascal.bianchi@telecom-paris.fr}\\
  \addr LTCI, Télécom Paris, Institut Polytechnique de Paris, France}

\editors{Sinno Jialin Pan and Masashi Sugiyama}

\begin{document}

\maketitle

\begin{abstract}
  Although \adam\ is a very popular algorithm for optimizing the weights of neural networks,
  it has been recently shown that it can diverge even in simple convex optimization examples.
  Several variants of \adam\ have been proposed to circumvent this
  convergence issue.
  In this work, we study the \adam\ algorithm for smooth nonconvex optimization under
  a boundedness assumption on the adaptive learning rate.
  The bound on the adaptive step size depends on the Lipschitz constant of the
  gradient of the objective function and provides safe theoretical adaptive
  step sizes.
  Under this boundedness assumption, we show a novel first order convergence rate result in both deterministic
  and stochastic contexts. Furthermore, we establish convergence rates of the function value sequence
  using the Kurdyka-\L{}ojasiewicz property.\\
\end{abstract}

\begin{keywords}
Nonconvex optimization, Adaptive gradient methods, Kurdyka-\L{}ojasiewicz inequality.
\end{keywords}

\section{Introduction}

Consider the unconstrained optimization problem
$ \min_{x \in \bR^d} f(x)$,
where $f:\bR^d \to \bR$ is a differentiable map and $d$ is an integer. Gradient descent is one of the most classical algorithms to solve this problem.
Since the seminal work \citet{robbins1951stochastic}, its stochastic counterpart became one of the most popular algorithms to solve machine learning problems (see \citet{bottou2018optimization} for a recent survey).
Recently, a class of algorithms called adaptive algorithms which are variants of stochastic gradient descent became very popular in machine learning applications \citep{duchi2011adaptive}. Using a coordinate-wise step size computed using past gradient information, the step size is adapted to the function to optimize and does not follow a predetermined step size schedule. Among these adaptive algorithms, \adam\, \citep{kingma2014adam} is very popular for optimizing the weights of neural networks. However, recently, \citet{j.2018on} exhibited a simple convex stochastic optimization problem over a compact set where \adam\ fails to converge because of its short-term gradient memory. Moreover, they proposed an algorithm called \textsc{Amsgrad} to fix the convergence issue of \adam\,. This work opened the way to the emergence of other variants of \adam\, to overcome its convergence issues (see \Cref{sec:related_works} for a detailed review). In this work, under a bounded step size assumption, we propose a theoretical analysis of \adam\, for nonconvex optimization.\\

\noindent\textbf{Contributions.}
\begin{itemize}
\item We establish a convergence rate for \adam\ in the deterministic case for nonconvex optimization under a bounded step size. This algorithm can be seen as a deterministic clipped version of \adam\, which guarantees safe theoretical step sizes. More precisely, if $n$ is the number of iterations of the algorithm, we show
a $O(1/n)$ convergence rate of the minimum of the squared  gradients norms by introducing a suitable Lyapunov function.

\item We show a similar convergence result for nonconvex stochastic optimization up to the limit of the variance of stochastic gradients under an almost surely bounded step size. In comparison to the literature, the hypothesis of the boundedness of the gradients is relaxed and the convergence result is independent of the dimension $d$ of the parameters.

\item We propose a convergence rate analysis of the objective function of the algorithm using the Kurdyka-\L{}ojasiewicz (K\L{}) property.
To the best of our knowledge, this is the first time such a result is established for an adaptive optimization algorithm.
\end{itemize}

The rest of the paper is organized as follows. \Cref{sec:algo} introduces the algorithm we analyze. \Cref{sec:related_works} considers some related works. \Cref{sec:convrate} establishes first order convergence rates in terms of the minimum of the gradients norms in both deterministic and stochastic settings. Finally, \Cref{sec:KLrates} derives function value convergence rates under the K\L{} property.
All the proofs are deferred to the Appendix in the supplementary material.

\section{A Momentum Algorithm with Adaptive Step Size}
\label{sec:algo}

\noindent \textbf{Notations.} All operations between vectors of $\bR^d$ are to read coordinatewise.
In particular, for two vectors $x$, $y$ in $\bR^d$ and $\alpha \in \ZZ$, we denote by $xy$, $x/y$, $x^\alpha$ the vectors on $\bR^d$ whose $k$-th coordinates are respectively given by $x_k y_k$, $x_k/y_k$, $x_k^\alpha$.
The vector of ones of $\bR^d$ is denoted by $\bf{1}$. When a scalar is added to a vector, it is added to each one of its coordinates. Inequalities are also to be read coordinatewise. If $x \in \bR^d, x \leq \lambda \in \bR$ means that each coordinate of $x$ is smaller than $\lambda$.\\

We investigate the following algorithm defined by two sequences $(x_n)$ and $(p_n)$ in $\bR^d$:
\begin{equation}
  \label{algo}
\begin{cases} x_{n+1} = x_{n} - a_{n+1} p_{n+1} \\  p_{n+1} = p_{n} + b\,(\nabla f(x_{n}) - p_n)\,\end{cases}
\end{equation}
where $\nabla f(x)$ is the gradient of $f$ at point $x$, $(a_n)$ is a sequence of vectors in $\bR^d$ with positive coordinates, $b$ is a positive real constant and $x_0, p_0 \in \bR^d$.

Algorithm~(\ref{algo}) includes the classical Heavy-ball method as a special case, but is much more general. Indeed, we allow the sequence of step sizes $(a_n)$ to be adaptive :  $a_n \in \bR^d$ may depend on the past gradients $g_k \eqdef \nabla f(x_k)$ and the iterates $x_k$ for $k \leq n$. We stress that the step size $a_n$ is a vector of $\bR^d$ and that the product $a_{n+1}p_{n+1}$ in (\ref{algo}) is read componentwise (this is equivalent to the formulation with a diagonal matrix preconditioner applied to the gradient \citep{McMahan2010AdaptiveBO,gupta2017unified,pmlr-v97-agarwal19b,staib19escapingsaddlesadaptive}).

We present in \Cref{table_algos} how to recover some of the famous algorithms with a vector step size formulation.
\begin{table*}[t]
\caption{Some famous algorithms.}\label{table_algos}

\centering
\resizebox{0.5\vsize}{!}{
\begin{tabular}{|c|c|c|}
  \hline
  \textbf{Algorithm} & \textbf{Effective step size} $a_{n+1}$ &   \textbf{Momentum}\\
  \hline
  \makecell{\textsc{SGD}\\ \citep{robbins1951stochastic}} & $a_{n+1} \equiv a$ & \makecell{$b=1$\\(no momentum)}\\
  \hline
  \makecell{\textsc{Adagrad}\\ \citep{duchi2011adaptive}} & $a_{n+1} = a\left(\sum_{i=0}^n g_i^2\right)^{-1/2}$ & $b=1$\\
  \hline
  \makecell{\textsc{Rmsprop}\\ \citep{tieleman2012lecture}} & $a_{n+1} = a\left[ \epsilon +\left(c\sum_{i=0}^n (1-c)^{n-i} g_i^2\right)^{1/2}  \right]^{-1}$ & $b=1$\\
  \hline
  \makecell{\adam\ \\\citep{kingma2014adam}} & $a_{n+1} = a\left[ \epsilon +\left(c\sum_{i=0}^n (1-c)^{n-i} g_i^2\right)^{1/2}  \right]^{-1}$ & \makecell{$0 \leq b \leq 1$\\(close to 0)}\\
  \hline
\end{tabular}
}
\end{table*}
In particular, \textsc{adam} \citep{kingma2014adam} defined by the iterates :
\begin{equation}
  \label{adam}
\begin{cases} x_{n+1} = x_{n} - \frac{a}{\epsilon+\sqrt{v_{n+1}}} p_{n+1} \\
              p_{n+1} = p_{n} + b\,(\nabla f(x_{n}) - p_n)\\
              v_{n+1} = v_n + c\,(\nabla f(x_{n})^2 - v_n)\,\end{cases}
\end{equation}
for constants $a\in \bR_+$, $b,c \in [0,1]$, can be seen as an instance of this algorithm by setting $a_n = \frac{a}{\epsilon + \sqrt{v_n}}$ where the vector $v_n$, as defined above, is an exponential moving average of the gradient squared. For simplification, we omit bias correction steps for $p_{n+1}$ and $v_{n+1}$. Their effect vanishes quickly along the iterations.

We introduce the main assumption on the objective function which is standard in gradient-based algorithms analysis.

\begin{assumption} \label{hyp:model}
The mapping $f:\bR^d \to \bR$ is:
  \begin{enumerate}[{\sl (i)},topsep=0pt,noitemsep]
  \item continuously differentiable and its gradient $\nabla f$ is $L-$Lipschitz continuous,
  \item bounded from below, i.e.,  $\inf_{x \in \bR^d} f(x) > -\infty$\,.
  \end{enumerate}
\end{assumption}

\section{Related Works}
\label{sec:related_works}

\subsection{The Heavy-Ball Algorithm.}

\noindent\textbf{Adaptive algorithms as Heavy Ball.}
Thanks to its small per-iteration cost and its acceleration properties (at least in the strongly convex case), the Heavy-ball method, also called gradient descent with momentum, recently regained popularity in large-scale optimization \citep{sutskever2013importance}. This speeding up idea dates back to the sixties with the seminal work of \citet{polyak1964some}. In order to tackle nonconvex optimization problems, \citet{ipiano} proposed iPiano, a generalization of the well known heavy-ball in the form of a forward-backward splitting algorithm with an inertial force for the sum of a smooth possibly nonconvex and a convex function. In the particular case of the Heavy-ball method, this algorithm writes for two sequences of reals $(\alpha_n)$ and $(\beta_n)$:
\begin{equation}
x_{n+1} = x_n - \alpha_n \nabla f(x_n) + \beta_n (x_{n}-x_{n-1})\,.
\label{hb}
\end{equation}

We remark that Algorithm~(\ref{algo}) can be written in a similar fashion by choosing step sizes $\alpha_n = b a_{n+1}$ and inertial parameters $\beta_n = (1-b)a_{n+1}/a_n$.
\citet{ipiano} only consider the case where $\alpha_n$ and $\beta_n$ are real-valued. Moreover, the latter does not consider adaptive step sizes, i.e step sizes depending on past gradient information. We can show some improvement with respect to \citet{ipiano} with weaker convergence conditions in terms of the step size of the algorithm (see \Cref{appendix:comp_ipiano}) while allowing adaptive vector-valued step sizes $a_n$ (see Proposition~\ref{prop:cv-a_n}).

It is shown in \citet{ipiano} that the sequence of function values converges and that every limit point is a critical point of the objective function. Moreover,
supposing that the Lyapunov function has the K\L{} property at a cluster point, they show the finite length of the sequence of iterates and its global convergence to a critical point of the objective function. Similar results are shown in \citet{wu2019general} for a more general version than iPiano \citep{ipiano} computing gradients at an extrapolated iterate like in Nesterov's acceleration.\\

\noindent\textbf{Convergence rate.}
\citet{ipiano} determines a $O(1/n)$ convergence rate (where $n$ is the number of iterations of the algorithm) with respect to the proximal residual which boils down to the gradient for noncomposite optimization. Furthermore, a recent work introduces a generalization of the Heavy-ball method (and Nesterov's acceleration) to constrained convex optimization in Banach spaces and provides a non-asymptotic hamiltonian based analysis with $O(1/n)$ convergence rate \citep{diakonikolas2019generalized}.
In the same vein, in \Cref{sec:convrate}, we establish a similar convergence result for an adaptive step size instead of a fixed predetermined step size policy like in the Heavy-ball algorithm (see \Cref{thm:deterministic}).\\

\noindent\textbf{Convergence rates under the K\L{} property.}
The K\L{} property is a powerful tool to analyze gradient-like methods. We elaborate on this property in \Cref{sec:KLrates}. Assuming that the objective function satisfies this geometric property, it is possible to derive convergence rates. Indeed, some recent progress has been made to study convergence rates of the Heavy-ball algorithm in the nonconvex setting. \citet{ochs2018local} establishes local convergence rates for the iterates and the function values sequences under the K\L{} property. The convergence proof follows a general method that is often used in non-convex optimization convergence theory. This framework was used for gradient descent \citep{absil2005convergence}, for proximal gradient descent (see \citet{attouch2009convergence} for an analysis with the \L{}ojasiewicz inequality) and further generalized to a class of descent methods called \textit{gradient-like descent} algorithms. 

K\L{}-based asymptotic convergence rates were established for constant Heavy-ball parameters \citep{ochs2018local}. Asymptotic convergence rates based on the K\L{} property were also shown \citep{johnstone2017convergence} for a general algorithm solving nonconvex nonsmooth optimization problems called Multi-step Inertial Forward-Backward splitting \citep{liang2016multi} which has iPiano and Heavy-ball methods as special cases. In this work, step sizes and momentum parameter vary along the algorithm run and are not supposed constant. However, specific values are chosen and consequently, their analysis does not encompass adaptive step sizes i.e. stepsizes that can possibly depend on past gradient information. In the present work, we establish similar convergence rates for methods such as \textsc{adam} under a bounded step size assumption (see~\Cref{thm:rates}). We also mention \cite{li2017convergence} which analyzes the accelerated proximal gradient method for nonconvex programming (\textsc{APG}nc) and establishes convergence rates of the function value sequence by exploiting the K\L{} property. This algorithm is a descent method i.e. the function value sequence is shown to decrease over time. In the present work, we analyze adaptive algorithms which are not descent methods. Note that even Heavy-ball is not a descent method. Hence, our analysis requires additional treatments to exploit the K\L{} property : we introduce a suitable Lyapunov function which is not the objective function. We also point out the recent work \citet{xie2019linear} which analyzes the \textsc{AdaGrad-Norm} algorithm under the global Polyak-\L{}ojasiewicz condition. This condition is a particular case of the K\L{} property (see \Cref{sec:KLrates}).\\

\noindent\textbf{Theoretical guarantees for \textsc{Adam}-like algorithms.}
The recent literature on adaptive optimization algorithms is vast.
For instance, for \textsc{AdaGrad}-like algorithms, several works cover
the 
nonconvex setting \citep{wu2018wngrad,ward2019adagrad,xie2019linear,liorabona19adagrad}.
In the following, we almost exclusively focus on \textsc{Adam}-like algorithms which are
different because of the momentum.
The first type of convergence results uses the online
optimization framework which controls the convergence rate of the average regret.
This framework was adopted for \textsc{AmsGrad, AdamNC} \citep{j.2018on}, \textsc{AdaBound}
and \textsc{AmsBound} \citep{luo2018adaptive}.
In this setting, it is assumed that the feasible set containing
the iterates is bounded by adding a projection step to the algorithm if needed.
We do not make such an assumption in our analysis. \citep{j.2018on} establishes
a regret bound in the convex setting.

The second type of theoretical results is based on the control of the norm
of the (stochastic) gradients.
We remark that some of these results depend on the dimension of the parameters.
\citet{zhou2018convergence} improves this dependency in comparison to \citet{chen2018convergence}.
The convergence result in \citet{basu2018convergence} is established under quite
specific values of $a_{n+1}, b_n$ and $\epsilon$.
\citet{zaheer2018adaptive} show a $O(1/n)$ convergence rate for an increasing mini-batch
size. However, the proof is provided for \textsc{RMSprop} and seems difficult to adapt
to \adam\ which involves a momentum term. Indeed, unlike \textsc{RMSProp}, \adam\ does not admit the objective function
as a Lyapunov function.

We also remark that all the available theoretical results assume boundedness of
the (stochastic) gradients. We do not make such an assumption.
Furthermore, we do not add any decreasing $1/\sqrt{n}$ factor in front of
the adaptive step size as it is considered in \citet{j.2018on,luo2018adaptive} and
\citet{chen2018convergence}. Although constant hyperparameters $b$ and $c$ are
used in practice, theoretical results are often established for non constant $b_n$ and $c_n$
\citep{j.2018on,luo2018adaptive}.
We also mention that most of the theoretical bounds depend on the dimension of the parameter
\citep{j.2018on,zhou2018convergence,chen2018closing,zou2019sufficient,chen2018convergence,luo2018adaptive}.\\

\noindent\textbf{Other variants of \adam\,.}
Recently, several other algorithms were proposed in the literature to enhance \adam\,.
Although these algorithms lack theoretical guarantees,
they present interesting ideas and show good practical performance. For instance,
\textsc{AdaShift} \citep{zhou2018adashift} argues that the convergence issue of \adam\ is due
to its unbalanced step sizes.
To solve this issue, they propose to use temporally shifted gradients to compute
the second moment estimate in order to decorrelate it from the first moment estimate.
\textsc{Nadam} \citep{dozat2016incorporating} incorporates Nesterov's acceleration into \adam\, in order
to improve its speed of convergence. Moreover, originally motivated by variance reduction,
\textsc{QHAdam} \citep{ma2018quasihyperbolic} replaces both \adam's moment estimates
by quasi-hyperbolic terms and recovers \adam\,, \textsc{Rmsprop} and \textsc{Nadam}
as particular cases (modulo the bias correction).
Guided by the same variance reduction principle, \textsc{Radam} \citep{liu2019variance}
estimates the variance of the effective step size of the algorithm
and proposes a multiplicative variance correction to the update rule.\\

\noindent\textbf{Step size bound.}
Perhaps, the closest idea to our algorithm is the recent \textsc{AdaBound}
\citep{luo2018adaptive} which considers a dynamic learning rate bound.
\citet{luo2018adaptive} show that extremely small and large learning rates can
cause convergence issues to \textsc{adam} and exhibit empirical situations where
such an issue shows up. Inspired by the gradient clipping
strategy proposed in \citet{pascanu2013difficulty} to tackle the problem of
vanishing and exploding gradients in training recurrent neural networks
(see \citet{zhang2019gradientclipping} for recent progress),
\citet{luo2018adaptive} apply clipping to the effective step size of the algorithm
in order to circumvent step size instability. More precisely, authors propose
dynamic bounds on the learning rate of adaptive methods
such as \textsc{Adam} or \textsc{AmsGrad} to solve the problem of extreme
learning rates which can lead to poor performance. Initialized respectively
at $0$ and $\infty$, lower and upper bounds both converge smoothly to a constant
final step size following a predetermined formula defined by the user. Consequently,
the algorithm resembles an adaptive algorithm in the first iterations and becomes
progressively similar to a standard \textsc{SGD} algorithm.
Our approach is different : we propose a static bound on the adaptive learning rate
which depends on the Lipschitz constant of the objective function.
This bound stems naturally from our theoretical derivations.

\section{First Order Convergence Rate}
\label{sec:convrate}

\subsection{Deterministic setting}
 Let $(H_n)_{n\geq0}$ be a sequence defined for all $n \in \bN$ by
$H_n \eqdef f(x_n)+ \frac{1}{2b} \ps{a_n,p_n^2}\,.$

We further assume the following step size growth condition.
\begin{assumption}
  \label{hyp:stepsize}
There exists $\alpha > 0$ s.t. $a_{n+1} \leq \frac{a_n}{\alpha}$.
\end{assumption}

Note that this assumption is satisfied for $\textsc{adam}$ with $\alpha = \sqrt{1-c}$ where $c$ is the parameter in~(\ref{adam}).
Unlike in \textsc{AmsGrad} \citep{j.2018on}, the step size $a_n$ is not necessarily nonincreasing. Indeed, $\alpha$ can be strictly
smaller than $1$ in \Cref{hyp:stepsize} as it is the case for \adam\,.

We provide a proof of the following key lemma in \Cref{proof:lemma_descent}.
\begin{lemma}
  \label{lemma:lyap}
Let \Cref{hyp:model,hyp:stepsize} hold true. Then, for all $n \in \bN$, for all $u \in \bR_{+}$,
\begin{equation}
  \label{eq:descent}
H_{n+1} \leq H_n - \ps{a_{n+1}p_{n+1}^2,A_{n+1}} - \frac{b}{2}\ps{a_{n+1}(\nabla f(x_n)-p_n)^2, B \mathbf{1}}\,,
\end{equation}
\begin{align*}
\text{where} \,\,
A_{n+1} &\eqdef 1-\frac{a_{n+1}L}{2} - \frac{|b-(1-\alpha)|}{2u} - \frac{1-\alpha}{2b}\,,
\text{and} \,\, B &\eqdef 1-  \frac{|b-(1-\alpha)|u}{b} - (1-\alpha)\,.
\end{align*}
\end{lemma}

\noindent We now state one of the principal convergence results about Algorithm~\ref{algo}. In particular, we establish a sublinear convergence rate for the minimum of the gradients norms until time~$n$.
\begin{theorem}
  \label{thm:deterministic}
  Let \Cref{hyp:model,hyp:stepsize} hold true. Suppose that $1-\alpha < b \leq 1$.
  Let $\varepsilon >0$ s.t. $a_{\sup} \eqdef \frac{2}{L}\left( 1- \frac{(b- (1-\alpha))^2}{2b\alpha} - \frac{1-\alpha}{2b} - \varepsilon\right)$ is nonnegative.
  Let $\delta >0$ s.t. for all $n \in \bN$,
  \begin{equation}
    \label{hyp:stepsize_bound}
    \delta \leq a_{n+1} \leq \min \left(a_{\sup},\frac{a_n}{\alpha}\right)\,.
  \end{equation}
  Then, the sequence $(H_n)$ is nonincreasing and $\sum_n \|p_n\|^2 < \infty$.
  In particular, $\lim x_{n+1}-x_n \to 0$ and $\lim \nabla f(x_n) \to 0$ as $n \rightarrow +\infty$.
  Moreover, for all $n\geq 1$,
  \begin{equation*}
    \min_{0\leq k \leq n-1} \|\nabla f(x_k)\|^2 \leq \frac{4}{nb^2} \left(\frac{H_0 - \inf f}{\delta \varepsilon}+\|p_0\|^2\right).
  \end{equation*}

\end{theorem}

\noindent\textbf{Sketch of the proof.} The key element of the proof is \lemmaref{lemma:lyap}
which is a descent lemma on the function $H$. Indeed, the assumptions of the theorem
guarantee that $A_{n+1} \geq \varepsilon$ and $B \geq 0$. Then, the result stems from summing
the inequalities of \lemmaref{lemma:lyap}. The proof can be found in \Cref{proof:deterministic}.\\

We provide some comments on this result.

\noindent\textbf{Dimension dependence.} Unlike most of the theoretical results for variants of \adam\ as gathered in
\Cref{tab:theoretical_guarantees}, we remark that the bound does not depend on the dimension $d$ of the parameter $x_k$.\\

\noindent\textbf{Comparison to gradient descent.} A similar result holds for deterministic gradient descent (see \citet[p.28]{nesterov2004book}).
If $\gamma$ is a fix step size
for gradient descent and there exist $\delta>0, \varepsilon>0$ s.t. $\gamma > \delta$ and $1- \frac{\gamma L}{2} > \varepsilon$, then (see \Cref{proof:grad_descent}) for all $n\geq 1$:
\begin{equation*}
 \min_{0\leq k \leq n-1} \|\nabla f(x_k)\|^2 \leq \frac{f(x_0) - \inf f}{n \gamma (1- \frac{\gamma L}{2})}
  \leq  \frac{f(x_0) - \inf f}{n\delta \varepsilon}.
\end{equation*}
When $p_0=0$ (this is the case for \adam\,), the bound in \Cref{thm:deterministic}
coincides with the gradient descent bound, up to the constant $4/b^2$.
We mention however that $\varepsilon$ for Algorithm~(\ref{algo}) is defined by a
slightly more restrictive condition than for gradient descent : when $b=1$,
there is no momentum and $a_{\sup} = \frac{1}{L}(1-2\varepsilon) < 2/L$.
Hence, under the boundedness of the effective step size, the algorithm has
a similar convergence guarantee to gradient descent.
Remark that the step size bound almost matches the classical $2/L$ upperbound on the step size
of gradient descent (see for example \citet[Theorem~2.1.14]{nesterov2004book}).
\\

\noindent\textbf{Stepsize bound.} Condition~\ref{hyp:stepsize_bound} should be seen as a clipping step of the algorithm.
Indeed, the lower bound on the effective stepsize  has not to be verified a posteriori after running the algorithm.
Instead, a clipping of the learning rate would ensure that this boundedness assumption holds. Furthermore, if we drop
the lower bound assumption on the effective step size $a_n$ from \Cref{thm:deterministic}, we still get the following result
(see \Cref{prop:cv-a_n}), for all $n\geq 1$,
\begin{equation*}
  \frac{1}{n}\sum_{k=0}^{n-1} \ps{a_{k+1},\nabla f(x_k)^2} \leq \frac{2(1+\alpha)}{n b^2\alpha} \left(\frac{H_0 - \inf f}{\varepsilon}+ \ps{a_0, p_0^2}\right).
\end{equation*}

\noindent\textbf{Influence of $\varepsilon$ and $\delta$.}
In the specific case of \textsc{Adam}, we obtain $La_{\sup}/2 + \varepsilon=0.93$ with the recommended default parameters $b= 0.1$ and $c = 0.001$. Hence, we can choose $\varepsilon$ of the order of $0.1$ without exceeding $0.93$. In view of Equation~\eqref{hyp:stepsize_bound_stoch}, the smaller is $\varepsilon$ and the larger will be the stepsizes. However, a small $\varepsilon$ deteriorates the bounds of Theorems~\ref{thm:deterministic} and~\ref{thm:stochastic}. Once $b$, $c$ (and then $\alpha$) are fixed, $\varepsilon$ can be seen as a constant. The clipping parameter $\delta$ can also be seen as constant once it is chosen.

\subsection{Stochastic setting}
We establish a similar bound in the stochastic setting.
Note that the control of the minimum of the gradients norms is also standard in nonconvex
stochastic optimization literature (see for example \citet{ghadimi2013stochastic}). 
Let $(\Xi,\mathfrak{S})$ denote a measurable space and $d \in \bN$.
Consider the problem of finding a local minimizer of the expectation
$F(x)\eqdef\bE(f(x,\xi))$ w.r.t. $x\in \bR^d$, where $f:\bR^d\times \Xi\to \bR$ is
a measurable map and $f(\,.\,,\xi)$ is a possibly nonconvex function depending
on some random variable~$\xi$. The distribution of $\xi$ is assumed to be unknown,
but revealed online by the observation of iid copies $(\xi_n:n\geq 1)$ of the r.v. $\xi$.
For a fixed value of $\xi$, the mapping $x\mapsto f(x,\xi)$ is
supposed to be differentiable, and its gradient w.r.t. $x$ is denoted
by $\nabla f(x,\xi)$. We study a stochastic version of Algorithm~(\ref{algo}) by
replacing the deterministic gradient $\nabla f(x_n)$ by $\nabla f(x_n, \xi_{n+1})$.

\begin{theorem}
  \label{thm:stochastic}
  Let \Cref{hyp:model} (for $F$) and \Cref{hyp:stepsize} hold true. Assume the following bound on the variance in stochastic gradients:
  $\bE \|\nabla f(x,\xi)-\nabla F(x)\|^2  \leq \sigma^2$ for all $x \in \bR^d$. Suppose moreover that $1-\alpha < b \leq 1$.
  Let $\varepsilon >0$ s.t. $\bar a_{\sup} \eqdef \frac{2}{L}\left( \frac{3}{4}- \frac{(b- (1-\alpha))^2}{2b\alpha} - \frac{1-\alpha}{2b} - \varepsilon\right)$ is nonnegative.
  Let $\delta >0$ s.t. for all $n \geq 1$, almost surely,
  \begin{equation}
    \label{hyp:stepsize_bound_stoch}
    \delta \leq a_{n+1} \leq \min \left(\bar a_{\sup},\frac{a_n}{\alpha}\right)\,.
  \end{equation}
  Then,
  \begin{equation*}
  \bE[\|\nabla F(x_\tau)\|^2] \leq \frac{4}{n b^2} \left(\frac{H_0 - \inf f}{\delta \varepsilon}+\|p_0\|^2\right) + \frac{4\bar a_{\sup}}{\delta \varepsilon b^2} \sigma^2,
\end{equation*}
where $x_\tau$ is an iterate uniformly randomly chosen from $\{x_0, \cdots, x_{n-1}\}$.
\end{theorem}

\begin{remark}
We recover the deterministic bound of \Cref{thm:deterministic}
when the gradients are noiseless ($\sigma = 0$).
The complete proof is deferred to \Cref{proof:stochastic}.
\end{remark}

Before proceeding, a few remarks are in order.

\noindent\textbf{SGD as a particular case.} By setting $b=1$ (no momentum) and $a_{n+1} = a_n$ for all
$n$ which implies $\alpha =1$, we recover a known rate for nonconvex \textsc{SGD} \citep{ghadimi2013stochastic} 
with a maximal stepsize here of $\bar a_{\sup} = \frac{1}{2L}(1-2\varepsilon)$ and note that the proof
can be slightly modified to make  $\bar a_{\sup}$ as close as possible to $1/L$. We highlight though
that the Lyapunov function $H$ was especially tailored to handle a momentum algorithm
and an analysis with $f$ as a Lyapunov function is largely satisfying for \textsc{SGD}.\\

\noindent\textbf{\textsc{RMSProp}.}
In the particular case where there is no momentum in the algorithm (i.e. \textsc{RMSProp}) and assuming that
the gradients are bounded, a similar convergence rate is obtained in \citet[Thm. 1]{zaheer2018adaptive}
(see \Cref{variants_adam}). Furthermore, although we assume boundedness of the step size
by Condition~(\ref{hyp:stepsize_bound_stoch}), we do not suppose
that $a_1 \leq \frac{\epsilon}{2L}$ (see table in \Cref{variants_adam}).
The latter assumption imposes a very small step size
($\epsilon = 10^{-8}$ in \citet{kingma2014adam}) which may result in a slow convergence.\\

\noindent\textbf{Stepsize lower bound.}
In the case of \adam\,($a_n = \frac{a}{\epsilon + \sqrt{v_n}}$), the uniform
lower bound $a_{n+1} \geq \delta$ prevents the exponential moving average $v_n$ of the
squared gradients from exploding. This can be guaranteed on the fly by a clipping of $a_n$.
If we drop the uniform lower bound on the effective step size,
we still obtain the following result (see Appendix.~\Cref{remark:without_lower_bound})
\begin{equation*}
\bE\left[\sum_{k=0}^{n-1}  \ps{a_{k+1}, \nabla f(x_k,\xi_{k+1})^2}\right] \leq \frac{2(1+\alpha)}{b^2 \alpha} \left(\frac{H_0 - \inf f}{\varepsilon} + \ps{a_0,p_0^2} + \frac{n \bar a_{\sup} \sigma^2}{\varepsilon}\right)\,.
\end{equation*}

\noindent\textbf{Influence of the momentum parameter.}
Note that $\varepsilon$ depends on the momentum parameter $b$ and consequently
the bound does not decrease with $b$. The influence of this parameter is more complex.

\section{Convergence Analysis under the K\L{} Property}
\label{sec:KLrates}

Historically introduced by the fundamental works of \citet{lojasiewicz1963propriete} and \citet{kurdyka1998gradients},
the K\L{} inequality is the key tool of our analysis. We refer to \citet{bolte2010characterizations} for an in-depth presentation of this property.
The K\L{} inequality is satisfied by a broad class of functions including most nonsmooth deep neural networks.
More precisely, as exposed in \citet[Section~5.2, Corollary~5.11]{davis2018stochastic} and \citet[Section~2.2]{castera2019inertial},
feedforward neural networks with arbitrary number of layers of arbitrary dimensions, with activations such as sigmoid, ReLU, leaky ReLU, tanh, softplus (and many others), 
 with a loss function such as $l_p$ norm, hinge loss, logistic loss or cross entropy (and many others), belong to this class of so-called \emph{definable} functions in an \emph{o-minimal structure} \citep{kurdyka1998gradients,attouch2010proximal,davis2018stochastic}. We refer the interested reader to \citet[Section~3, Section~C]{zeng19aglobalconvergenceBCD} for general conditions for which K\L{} inequality holds in the context of deep neural networks training models.
The class of \emph{definable} functions is stable under all the typical functional operations in optimization (e.g. sums, compositions, inf-projections) and generalizes the class of semialgebraic functions including objective functions such as $\|\cdot\|_p$ for $p$ rational, real polynomials, rank, etc. (see \citet[Appendix]{bolte2014proximal}).

The K\L{} inequality has been used to show the convergence of several first-order optimization methods towards critical points 
\citep{attouch2009convergence,attouch2010proximal,bolte2014proximal,li2017convergence}.
In this section, we use a methodology exposed in \citet[Appendix]{bolte2018first} to show convergence rates based on the K\L{} property.
Recently developed in \cite{bolte2014proximal}, this abstract convergence mechanism can be used for any \textit{descent} type algorithm.
We modify it to encompass momentum methods. Note that although this modification was initiated in \citet{ipiano,ochs2018local}, we use
a different separable Lyapunov function. The first part of the proof
follows these approaches and the second part follows the proof of \citet[Theorem 2]{johnstone2017convergence}.

Consider the function $H : \bR^d \times \bR^d \to \bR$ defined for all $z=(x,y) \in \bR^d \times \bR^d$ by
\begin{equation}
  \label{def:H}
  H(z) = H(x,y) = f(x) + \frac{1}{2b}\|y\|^2\,.
\end{equation}
Notice that $H_n = f(x_n) + \frac{1}{2b} \ps{a_n,p_n^2} = H(x_n,y_n)$ where $(y_n)_{n \in \bN}$ is defined for all $n \in \bN$ by $y_n = \sqrt{a_n} p_n$.

\noindent{\textbf{Notations and definitions.}} If $(E,\mathsf d)$ is a metric space, $z\in E$ and
$A$ is a non-empty subset of $E$, we use the notation
$\mathsf d(z,A) \eqdef \inf\{\mathsf d(z,z'):z'\in A\}$\,.
The set of critical points of the function $H$ is defined by
$\text{crit}\, H \eqdef \{ z \in \bR^{2d} \, \text{s.t.}\, \nabla H(z) = 0\}\,.$

\begin{assumption}
  \label{hyp:coercivity}
  $f$ is coercive, that is
   $
   f(x) \to +\infty \,\text{as}\, \|x\| \to +\infty.
   $
\end{assumption}

\Cref{hyp:coercivity} will be particularly useful to ensure that the sequence of the iterates
$(z_k)_{k \geq 0}$ of Algorithm~(\ref{algo}) is bounded. Indeed, a coercive function has compact
level sets and \lemmaref{lemma:lyap} will guarantee that the iterates lie in a level
 set of the function $H$.

 We now introduce the limit point set of the sequence $(z_k)_{k \geq 0}$ and
 exhibit some of its properties.

 \begin{definition}{\textbf{(Limit point set)}}
 The set of all limit points of $(z_k)_{k \in \bN}$ initialized at $z_0$ is defined by
   $$
   \omega (z_0) \eqdef \{ \bar z \in \bR^{2d} : \exists \, \text{an increasing sequence of integers}\, (k_j)_{j \in \bN} \, \text{s.t}\, z_{k_j} \to \bar z \,\text{as}\, j \to \infty \}\,.
   $$
 \end{definition}

\begin{lemma}{\textbf{(Properties of the limit point set)}}
\label{lemma:subseq}
 Let $(z_k)_{k \in \bN}$ be the sequence defined for all $k \in \bN$ by $z_k = (x_k,y_k)$
 where $y_k = \sqrt{a_k}p_k$ and $(x_k,p_k)$ is generated by Algorithm~(\ref{algo}) from a starting point $z_0$.
  Let \Cref{hyp:model,hyp:stepsize,hyp:coercivity} hold true. Assume that Condition~(\ref{hyp:stepsize_bound}) holds. Then,
 \begin{enumerate}[(i),noitemsep]
   \item $\omega(z_0)$ is a nonempty compact set.
   \item $\omega(z_0) \subset \text{crit} H = \text{crit} f \times \{0\}$\,.
   \item $\lim\limits_{k \rightarrow +\infty} \sd(z_k, \omega(z_0)) = 0$.
   \item $H$ is finite and constant on $\omega(z_0)$.
 \end{enumerate}
\end{lemma}

We introduce the K\L{} inequality in the following. Define
$[\alpha < H < \beta] \eqdef \{ z \in \bR^{2d} \,:\, \alpha < H(z) < \beta\}$\,.
Let $\eta >0$ and define
$\Phi_{\eta}$ as the set of continuous functions $\varphi$ on $[0,\eta)$ which are also
continuously differentiable on $(0, \eta)$, concave and satisfy $\varphi(0)=0$ and $\varphi' > 0$.

\begin{definition}{\textbf{(K\L{} property,\,\citet[Appendix]{bolte2018first})}}
\label{def:KL}
  A proper and lower semicontinuous (l.s.c) function $H : \bR^{2d} \to (-\infty, +\infty]$ has the K\L{} property locally at
  $\bar z \in \text{dom}\,H$ if there exist $\eta >0$, $\varphi \in \Phi_{\eta}$ and a neighborhood $U(\bar z)$ s.t.
  for all $z \in U(\bar z) \cap [H(\bar z) < H < H(\bar z) + \eta]$ :
  \begin{equation}
    \label{eq:kl}
        \varphi'(H(z)-H(\bar z))\,\|\nabla H(z)\| \geq 1\,.
  \end{equation}
\end{definition}

When $H(\bar z)= 0$, we can rewrite \Cref{eq:kl} as : $\|\nabla (\varphi \circ H)(z)\| \geq 1$ for suitable $z$ points. This means that $H$
becomes sharp under a reparameterization of its values through the so-called desingularizing function $\varphi$.

The function $H$ is said to be a K\L{} function if it has the K\L{} property at each point of the domain of its gradient.
Note that this property can be defined for nonsmooth functions using the Clarke subdifferential 
in order to encompass nonsmooth neural networks.
We limit ourserlves to the simpler differentiable setting.
K\L{} inequality holds at any non critical point (see \citet[Remark 3.2 (b)]{attouch2010proximal}).
We introduce now a uniformized version of the K\L{} property which will be useful for our analysis.

\begin{lemma}{\textbf{(Uniformized K\L{} property,\,\citet[Lemma 6, p 478]{bolte2014proximal})}}
  \label{lemma:unif_KL}
Let $\Omega$ be a compact set and let $H : \bR^{2d} \to (-\infty, +\infty]$ be a proper l.s.c function. Assume that $H$ is constant on $\Omega$ and
satisfies the K\L{} property at each point of $\Omega$. Then, there exist $\varepsilon >0, \eta >0$ and $\varphi \in \Phi_{\eta}$ such that for all
$\bar z \in \Omega$, for all $z \in \{z \in \bR^{d} : \, \sd(z,\Omega)<\varepsilon\} \cap [H(\bar z) < H < H(\bar z) + \eta]$, one has
\begin{equation}
\label{eq:uniformized-KL}
\varphi'(H(z)-H(\bar z))\|\nabla H(z)\| \geq 1
\end{equation}
\end{lemma}

\begin{definition}{\textbf{(K\L{} exponent)}}
  If $\varphi$ can be chosen as $\varphi(s) = \frac{\bar c}{\theta}s^\theta$ for some $\bar c > 0$ and $\theta \in (0,1]$\, in \Cref{def:KL},
  then we say that $H$ has the K\L{} property at $\bar z$ with an exponent of~$\theta$\,
  \footnote{$\alpha \eqdef 1-\theta$ is also defined as the K\L{} exponent in other papers \citep{li2018calculus}.}.
  We say that $H$ is a K\L{} function with an exponent $\theta$ if it has the same exponent $\theta$ at any~$\bar z$.
\end{definition}
In the particular case when $\theta = 1/2$, we recover the Polyak-\L{}ojasiewicz condition (see for example \citet{karimi2016linear}) satisfied for strongly convex functions.
Furthermore, if $H$ is a proper closed semialgebraic function, then $H$ is a K\L{} function with a suitable exponent $\theta \in (0,1]$\,.
The slope of $\varphi$ around the origin informs about the "flatness" of a function around a point. Hence, the
K\L{} exponent allows to obtain convergence rates. In the light of this remark, we state one of the main results of this work.

\begin{theorem}{\textbf{(Convergence rates)}}
  \label{thm:rates}
  Let $(z_k)_{k \in \bN}$ be the sequence defined for all $k \in \bN$ by $z_k = (x_k,y_k)$
  where $y_k = \sqrt{a_k}p_k$ and $(x_k,p_k)$ is generated by Algorithm~(\ref{algo}) from a starting point $z_0$.
  Let \Cref{hyp:model,hyp:stepsize,hyp:coercivity} hold true. Assume that Condition~(\ref{hyp:stepsize_bound}) holds.
  Suppose moreover that $H$ is a K\L{} function with K\L{} exponent $\theta$.
  Then, the sequence $(H(z_k))_{k \in \bN}$ converges to $f(x_{*})$ where $x_{*}$ is a critical point of~$f$ and the following convergence rates hold:
  \vskip -0.4in
  \begin{enumerate}[(i),noitemsep]
          \item If $\theta = 1$, then $f(x_k)$ converges in a finite number of iterations.
          \item If $1/2 \leq \theta < 1$, then $f(x_k)$ converges to $f(x_{*})$ linearly
          i.e. there exist $q \in (0,1), C>0$ \, s.t. $f(x_k)-f(x_{*}) \leq C\,q^k$\,.
          \item If $0 < \theta < 1/2$\,, then  $f(x_k)-f(x_{*}) = O(k^{\frac{1}{2\theta-1}})$\,.
    \end{enumerate}
\end{theorem}

The exact same rates hold for gradient descent by supposing that $f$ (instead of $H$) is KL with exponent $\theta$. Assumption~\ref{hyp:stepsize} and condition~(\ref{hyp:stepsize_bound}) are not needed in this case.\\

\noindent\textbf{Sketch of the proof.} The proof consists of two main steps. The first one is to show that the iterates
enter and stay in a region where the K\L{} inequality holds. This is achieved using the properties of the limit set (\lemmaref{lemma:subseq}) and
the uniformized K\L{} property (\lemmaref{lemma:unif_KL}). Then, the second step is to exploit this inequality to derive the sought convergence results.
We defer the complete proof to \Cref{proof:KL_rates}.

We introduce a lemma in order to make the K\L{} assumption on the objective function $f$
instead of the function $H$.

\begin{lemma}
  \label{lemma:f_KL}
Let $f$ be a continuously differentiable function satisfying the KL property at $\bar x$ with an exponent
of $\theta \in (0,1/2]$. Then the function $H$ defined in \Cref{def:H} has also the K\L{} property at
$(\bar x, 0)$ with an exponent of~$\theta$\,.
\end{lemma}

The following result derives a convergence rate on the objective function values under
a K\L{} assumption on this same function instead of an assumption on the Lyapunov function~$H$.
The result is an immediate consequence of \lemmaref{lemma:f_KL} and \Cref{thm:rates}.

\begin{corollary}
  Let $(z_k)_{k \in \bN}$ be the sequence defined for all $k \in \bN$ by $z_k = (x_k,y_k)$
  where $y_k = \sqrt{a_k}p_k$ and $(x_k,p_k)$ is generated by Algorithm~(\ref{algo}) from a starting point $z_0$.
  Let \Cref{hyp:model,hyp:stepsize,hyp:coercivity} hold true. Assume that Condition~(\ref{hyp:stepsize_bound}) holds.
  Suppose moreover that $f$ is a K\L{} function with K\L{} exponent $\theta \in (0,1/2)$.
  Then, the sequence $(H(z_k))_{k \in \bN}$ converges to $f(x_{*})$ where $x_{*}$ is a critical point of~$f$ and
  $f(x_k)- f(x_{*}) =  O(k^{\frac{1}{2\theta-1}})$\,.
\end{corollary}

\subsection{Toy problem : K\L{} rates for $f(x) = x^p$.}
\label{sec:simus}

\begin{figure*}
  \begin{center}
  \centerline{
  \includegraphics[width=\textwidth]{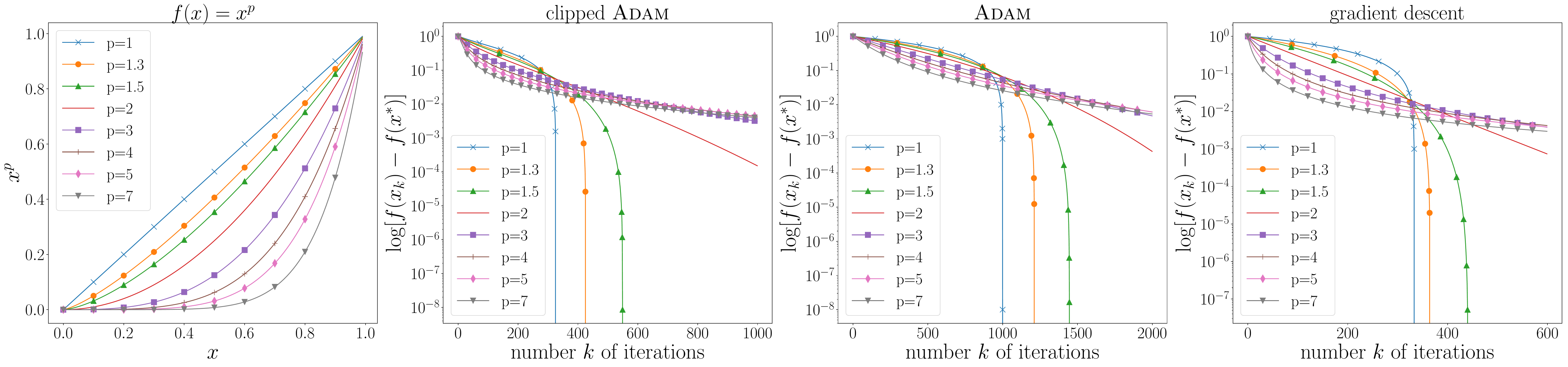}
  }
  \caption{Illustration of K\L{} rates for a simple objective function $f(x) = x^p$.
  From left to right : $(i)$ curves of $f(x) = x^p$, $(ii)$ clipped version of \adam\,(see Algorithm~(\ref{algo})),
  $(iii)$ \adam\,,$(iv)$ Gradient Descent. Best seen in color.}
  \label{fig:KL_rates}
  \end{center}
\vskip -0.4in
\end{figure*}

K\L{} rates are \textit{asymptotic} rates in the sense that the constants cannot be explicited in the convergence rates.
As a consequence, the rates can be hardly observable in practice from experiments. However, we can still illustrate these convergence results  (\Cref{thm:rates}) in a simple toy example to give more insight.
Consider the problem of
minimizing the function $f(x) = x^p$ for a real $p \in [1,7]$.
One can easily show that $f$ is a K\L{} function with K\L{} exponent $\theta =\frac{1}{p}$.
Note that the K\L{} exponent is difficult to compute in general.
This justifies the choice of this toy problem. Moreover, even if the function $f$ is indeed convex,
we recall the reader that the K\L{} property is a local geometric property of the function
that is only interesting at its critical points (since it is automatically verified at any non critical point).
Notice that the K\L{} analysis is valid in the general nonconvex case.
The present toy example remains relevant if we modify the objective function $f$ to be nonconvex and
still keep a $x^p$ shape in a neighborhood of the point zero which is the unique critical point
in this example.

The K\L{} exponent as shown in the first plot in \Cref{fig:KL_rates} encodes
information about the flatness of the function $f$. Indeed, as $p$ increases, the function $f$
gets flatter around the origin $x=0$.
We run the clipped version of \adam\ (see Algorithm~\ref{algo}), the \adam\ algorithm
and gradient descent on the functions $f$ corresponding to different values of the exponent $\theta$,
from the same initialization point $x=1$.
As expected from \Cref{thm:rates} for the clipped \adam\,, we observe in
\Cref{fig:KL_rates} that $f(x_k)$ converges linearly or even in a finite number
of iterations for $p \in \{ 1, 1.3, 1.5, 2\}$. Notice that the linear rate is clearly
observable for $p=2$ corresponding to $\theta = \frac{1}{2}$. Even if we did not
establish K\L{} rates for original \adam\,, \Cref{fig:KL_rates} shows that it presents a
very similar behavior to the clipped version of \adam\ in terms of K\L{} convergence rates
in this simple problem. We also represent gradient descent iterates for comparison.
Note that K\L{} rates are known to hold for gradient descent. Moreover, for $p>2$,
we also observe a slower rate corresponding to the sublinear rate of the function values.

\section{Conclusion}
\label{sec:conclusion}

In this paper, we provided convergence rates for a clipped version of \adam\ which
stems from a boundedness assumption on the effective stepsize of the original \adam\,. More precisely,
similarly to gradient descent, we established a $O(1/n)$ convergence rate of
the minimum of the squared gradient norms in the deterministic case.
Furthermore, we showed a similar convergence result in the stochastic setting up to
the variance of the noisy gradients. Finally, we established function value convergence rates
under the same boundedness assumption on the effective stepsizes together with the K\L{} geometric property.
This property is a powerful tool allowing to address nonconvex nonsmooth optimization and covers most deep neural networks.

\acks{
We thank the anonymous reviewers for their helpful comments.
A.B. was supported by the 'Futur \& Ruptures' research program which is jointly funded by the IMT, the Mines-Télécom Foundation and the Carnot TSN Institute.
}

\bibliography{acml20}

\newpage
\pagenumbering{arabic}

\appendix
\section{First Order Convergence Rate}

\subsection{Variants of \adam\,}
\label{variants_adam}
We list most of the existing variants of the \adam\ algorithm together with their theoretical convergence guarantees in \Cref{tab:variants_adam}.

\label{tab:theoretical_guarantees}
\begin{table}[htbp]
\caption{\textbf{Theoretical guarantees of variants of \adam\,.}
The gradient is supposed $L$-lipschitz continuous in all the convergence results. $g_{1:T,i} = [g_{1,i},g_{2,i},\cdots,g_{T,i}]^T$.
\label{tab:variants_adam}
}

\begin{adjustbox}{max width=1.33\textwidth,,angle=90}
\begin{tabular}{|l|c|c|c|c|c|c|l|c|}
  \hline
  \textbf{Algorithm} & \textbf{Effective step size} $a_{n+1}$ &  $b_n$ & $c_n$ & Assumptions & Convergence Result\\
  \hline
  \makecell{\textsc{AmsGrad}\tiny{$^{(1)}$}, \textsc{AdamNC}\tiny{$^{(2)}$}\\ \citep{j.2018on}} & \makecell{$\frac{a_0}{\sqrt{n}}\frac{1}{\sqrt{\hat v_n}}$\\ \tiny{$^{(1)}$}$\hat v_{n+1} = \max (\hat v_n,(1-c_n)v_n + c_n g_n^2)$\\ \tiny{$^{(2)}$}$\hat v_{n+1} = (1-c_n)\hat v_n + c_n g_n^2$} &\makecell{$1-b_1 \lambda^{n-1}\,$ \\or $1 - \frac{b_1}{n}$}& \makecell{$c_n \equiv c_1$\\$\frac{c_1}{n}$ \tiny{(for \textsc{AdamNC})}} & \makecell[l]{\tabitem convex functions \\ \tabitem bounded gradients\\ \tabitem bounded feasible set\\ \tabitem $\sum_{i=1}^d \hat v_{T,i}^{1/2} \leq d$ \tiny{(\textsc{AmsGrad})}\\ \tabitem $\sum_{i=1}^d \|g_{1:T,i}\|_2 \leq \sqrt{dT}$\\ \tabitem $b_1 < \sqrt{c_1}$\scriptsize{(\textsc{AdamNC})} } & \makecell{$R_T/T = O(\sqrt{log T/T})$\\ $\frac{R_T}{T} = O(1/\sqrt{T})$ \tiny{(\textsc{AdamNC})}}\\
  \hline
  \makecell{\textsc{Adam}\\ \citep{basu2018convergence}} & \makecell{$\frac{4\|g_n\|^2\eta}{3L(1-(1-b)^n)^2(\eta+2\sigma)^2}\frac{1}{\epsilon+\sqrt{v_n}}$\\ $v_{n+1} = (1-c_1)v_n + c_1 g_n^2$ } & \makecell{$b_n \equiv b_1$\\ $= 1 - \frac{\eta}{\eta + 2 \sigma}$} & $c_n \equiv c_1$ & \makecell[l]{\tabitem $\sigma$-bounded gradients\\ \tabitem $\epsilon = 2\sigma$} & \makecell{$\forall \eta >0\, \exists n \leq \frac{9L\sigma^2(f(x_2)-f(x_*))}{\eta^6}$ \\s.t. $\|g_n\| \leq \eta$}\\
  \hline
  \makecell{\textsc{Padam},\textsc{AmsGrad}\\ \citep{zhou2018convergence}}& \makecell{$\frac{1}{\sqrt{N}} \frac{1}{\hat v_n^p}$\\ $\frac{1}{\sqrt{dN}} \frac{1}{\hat v_n^{\frac{1}{2}}}$\tiny{(\textsc{AmsGrad})}\\ $\hat v_n = \max(\hat v_{n-1}, (1-c)v_{n-1}+c g_n^2)$} & $b_n \equiv b$ & $c_n \equiv c$ & \makecell[l]{\tabitem bounded gradients \\ For \textsc{Padam:} \tabitem $p \in [0,\frac{1}{4}]$ \\ \tabitem $1-b < (1-c)^{2p}$\\ \tabitem $\sum_{i=1}^d \|g_{1:N,i}\|_2 \leq \sqrt{d N}$\\ \tiny{\textsc{AmsGrad}:} $p=\frac{1}{2}$ and $1-b < 1-c$ }& \makecell{$\bE [ \|g_\tau\|^2] = O(\frac{1+\sqrt{d}}{\sqrt{N}}+ \frac{d}{N})$\\ $=O( \sqrt{\frac{d}{N}}+ \frac{d}{N})$ \tiny{(\textsc{AmsGrad})} \\ $\tau$ uniform r.v in $\{1,\cdots,N\}$ } \\
  \hline
  \makecell{\textsc{RmsProp}\tiny{$^{(1)}$}, \textsc{Yogi}\tiny{$^{(2)}$}\\ \citep{zaheer2018adaptive}} & \makecell{$\frac{a_1}{\epsilon + \sqrt{v_n}}$\\ \tiny{$^{(1)}$}$v_{n+1}=(1-c)v_n+c g_n^2$\\ \tiny{$^{(2)}$}$v_n = v_{n-1} - c \text{sign}(v_{n-1}-g_n^2)$} & $b_n \equiv b$ & $c_n \equiv c$ & \makecell[l]{\tabitem $G$-bounded gradients \\ \tabitem $a_1 \leq \frac{\epsilon\sqrt{1-c}}{2L}$ \scriptsize{(\textsc{Yogi})} \\ \tabitem $a_1 \leq \frac{\epsilon}{2L}$ \tabitem $c \leq \frac{\epsilon^2}{16 G^2}$\\ \tabitem $\sigma^2$-bounded variance}& \makecell{$\bE [ \|g_\tau\|^2] = O(\frac{1}{N}+\sigma^2)$\\ $\tau$ uniform r.v in $\{1,\cdots,N\}$\\ $O(\frac{1}{N})$ if minibatch $\Theta(N)$ } \\
  \hline
  \makecell{\textsc{AmsGrad}\tiny{$^{(1)}$}, \textsc{AdaFom}\tiny{$^{(2)}$}\\ \citep{chen2018convergence}} & \makecell{$\frac{1}{\sqrt{n}}\frac{1}{\sqrt{\hat v_n}}$\\ \tiny{$^{(1)}$} $\hat v_{n+1} = \max (\hat v_n,(1-c_n)v_n + c_n g_n^2)$ \\ \tiny{$^{(2)}$} $\hat v_{n+1} = (1-\frac{1}{n}) \hat v_n + \frac{1}{n} g_n^2)$} & non-increasing & $c_n \equiv c_1$ & \makecell[l]{\tabitem bounded gradients\\ \tabitem $\exists c >0$ s.t. $|g_{1,i}| \geq c$}  & $\underset{n \in [0,N]}{\min} \bE [ \|g_n\|^2]= O(\frac{log N + d^2}{\sqrt{N}})$\\
  \hline
  \makecell{\textsc{Generic Adam}\\ \citep{zou2019sufficient}} & \makecell{$\frac{\alpha_n}{\sqrt{v_n}}$\\ $v_{n+1} = (1-c_n)v_n + c_n g_n^2$\\ $\alpha_n = \hat\alpha \frac{\sqrt{1-(1-c)^n}}{1-(1-b)^n}$} & $b_n \geq b > 0$ & \makecell{$0<c_n<1$\\ \scriptsize{non-increasing}\\ \scriptsize{$\lim c_n = c > b^2$}}& \makecell[l]{\tabitem bounded gradients\\ in expectation\\ \tabitem $d_n \leq \frac{\alpha_n}{\sqrt{c_n}} \leq c_0 d_n$\\ $d_n$ non-increasing }& \makecell{$\bE [ \|g_\tau\|^{\frac{4}{3}}]^{\frac{3}{2}} \leq \frac{C + C'\sum_{n=1}^N\alpha_n \sqrt{c_n}}{N\alpha_N}$\\ $\tau$ uniform r.v in $\{1,\cdots,N\}$ }\\
  \hline
  \makecell{\textsc{AdaBound}\tiny{$^{(1)}$}, \textsc{AMSBound}\tiny{$^{(2)}$}\\ \citep{luo2018adaptive}} &\makecell{$\frac{1}{\sqrt{n}}\text{clip}(\frac{\alpha}{\sqrt{v_n}},\eta_l(n),\eta_u(n))$\\$\eta_l(n)$ non-decreasing to $\alpha_*$\\$\eta_u(n)$ non-increasing to $\alpha_*$\\ \tiny{$^{(1)}$}$v_{n+1} = (1-c)v_n + c g_n^2$ \\ \tiny{$^{(2)}$}$v_{n+1} = \max (v_n,(1-c)v_n + c g_n^2)$} & \makecell{ \scriptsize{$1-(1-b)\lambda^{n-1}$} \\or $1- \frac{1-b}{n}$ \\ $b_n\geq b$} & $c_n \equiv c$ & \makecell[l]{\tabitem bounded gradients \\ \tabitem closed convex\\ bounded feasible set\\ \tabitem $1-b < \sqrt{1-c}$} & $R_T/T = O(1/\sqrt{T})$ \\
  \hline
\end{tabular}
\end{adjustbox}
\end{table}

\begin{remark}
The average regret bound result in the last line of \Cref{tab:variants_adam}
figures in \cite{luo2018adaptive}. Actually, according to \citet{savarese2019adaboundissue},
slightly different assumptions on the bound functions should be considered to guarantee this regret rate.
\end{remark}

\subsection{Proof of \lemmaref{lemma:lyap}}
\label{proof:lemma_descent}

  Supposing that $\nabla f$ is $L-$Lipschitz, using Taylor's expansion and the expression of $p_n$ in the algorithm, we obtain the following inequality:
  \begin{equation}
  f(x_{n+1}) \leq f(x_n) - \ps{\nabla f(x_n),a_{n+1} p_{n+1}} + \frac{L}{2} \|a_{n+1} p_{n+1}\|^2
  \label{eq:fdiff}
  \end{equation}

Moreover,
\begin{equation}
  \frac{1}{2b} \ps{a_{n+1},p_{n+1}^2} - \frac{1}{2b} \ps{a_n,p_n^2} = \frac{1}{2b} \ps{a_{n+1},p_{n+1}^2-p_n^2} + \frac{1}{2b} \ps{a_{n+1}-a_n,p_n^2}.
\end{equation}

Observing that $p_{n+1}^2 - p_n^2 = - b^2(\nabla f(x_n)-p_n)^2 +2b p_{n+1} (\nabla f(x_n)-p_n)$, we obtain after simplification :
\begin{equation}
  H_{n+1} \leq H_n + \frac{L}{2} \|a_{n+1} p_{n+1}\|^2 - \frac{b}{2} \ps{a_{n+1}, (\nabla f(x_n)-p_n)^2} - \ps{a_{n+1}p_{n+1},p_n}+ \frac{1}{2b}\ps{a_{n+1}-a_n,p_n^2}.
\end{equation}

Using again $p_n = p_{n+1} - b(\nabla f(x_n)-p_n)$, we replace $p_n$ :

\begin{align*}
  H_{n+1}    &\leq H_n + \frac{L}{2} \|a_{n+1} p_{n+1}\|^2- \frac{b}{2} \ps{a_{n+1},(\nabla f(x_n)-p_n)^2}\\
             & - \ps{a_{n+1},p_{n+1}^2}  +b \ps{a_{n+1}p_{n+1},\nabla f(x_n)-p_n}+ \frac{1}{2b}\ps{a_{n+1}-a_n,p_n^2}.
\end{align*}

Under \Cref{hyp:stepsize}, we write: $\ps{a_{n+1}-a_n,p_n^2} \leq (1-\alpha)\ps{a_{n+1},p_n^2}$ and using $p_n^2 = p_{n+1}^2 + b^2(\nabla f(x_n)-p_n)^2 -2b p_{n+1} (\nabla f(x_n)-p_n)$, it holds that:

\begin{align*}
  \label{eq:Hn1}
H_{n+1} &\leq H_n - \ps{a_{n+1},p_{n+1}^2} - \frac{b}{2} \ps{a_{n+1},(\nabla f(x_n)-p_n)^2}\\
        &+ \frac{L}{2} \|a_{n+1} p_{n+1}\|^2 +(b-(1-\alpha)) \ps{a_{n+1}p_{n+1},\nabla f(x_n)-p_n}\\
        &+ \frac{1-\alpha}{2b} \ps{a_{n+1},p_{n+1}^2} +\frac{b(1-\alpha)}{2} \ps{a_{n+1},(\nabla f(x_n)-p_n)^2}.
\end{align*}

Using the classical inequality $xy \leq \frac{x^2}{2u} + \frac{uy^2}{2}$, we have :
\begin{equation}
  \label{eq:young1}
  (b-(1-\alpha)) a_{n+1}p_{n+1}(\nabla f(x_n)-p_n) \leq \frac{|b-(1-\alpha)|}{2u}\ps{a_{n+1},p_{n+1}^2} + \frac{|b-(1-\alpha)|u}{2}\ps{a_{n+1},(\nabla f(x_n)-p_n)^2}.
\end{equation}

Hence, after using this inequality and rearranging the terms, we derive the following inequality:

\begin{align*}
H_{n+1} &\leq H_n - \ps{a_{n+1}p_{n+1}^2, 1-\frac{a_{n+1}L}{2} - \frac{|b-(1-\alpha)|}{2u} - \frac{1-\alpha}{2b}}\\
        & - \frac{b}{2} \ps{a_{n+1}(\nabla f(x_n)-p_n)^2, \left( 1-  \frac{|b-(1-\alpha)|u}{b} - (1-\alpha)\right) \mathbf{1}}.
\end{align*}

This concludes the proof.

\subsection{A first result under an upperbound of the step size}
\label{proof:prop_cv-a_n}

\begin{proposition}
  \label{prop:cv-a_n}
Let \Cref{hyp:model} hold true. Suppose moreover that $1-\alpha < b \leq 1$.
Let $\varepsilon >0$ s.t. $a_{\sup} \eqdef \frac{2}{L}\left( 1- \frac{(b- (1-\alpha))^2}{2b\alpha} - \frac{1-\alpha}{2b} - \varepsilon\right)$ is nonnegative.
Assume for all $n \in \bN$,
\begin{equation*}
      a_{n+1} \leq \min \left(a_{\sup},\frac{a_n}{\alpha}\right)\,.
\end{equation*}
Then, for all $n\geq 1$,
\begin{equation*}
  \sum_{k=0}^{n-1} \ps{a_{k+1},\nabla f(x_k)^2} \leq  \frac{2(1+\alpha)}{b^2\alpha} \left(\frac{H_0 - \inf f}{\varepsilon}+ \ps{a_0, p_0^2}\right)
\end{equation*}

\end{proposition}

\begin{proof}
This is a consequence of \lemmaref{lemma:lyap}. Conditions $A_{n+1} \geq \varepsilon$ and $B \geq 0$ write as follow :
\begin{equation*}
  a_{n+1} \leq \frac{2}{L}\left( 1- \frac{b-(1-\alpha)}{2u} - \frac{1-\alpha}{2b} - \varepsilon \right) \quad \text{and} \quad u \leq \frac{\alpha b}{b-(1-\alpha)}\,.
\end{equation*}
We get the assumption made in the proposition by injecting the second condition into the first one and adding the assumption $\frac{a_{n+1}}{a_n} \leq \frac{1}{\alpha}$ made in the lemma.
Under this assumption, we sum over $0\leq k \leq n-1$ \Cref{eq:descent}, rearrange it and use $A_{n+1} \geq \varepsilon$, $B \geq 0$ to obtain :

\begin{equation*}
\sum_{k=0}^{n-1} \varepsilon \, \ps{a_{k+1},p_{k+1}^2} \leq H_0 - H_n\,,
\end{equation*}
Then, observe that $H_n \geq f(x_n) \geq \inf f$. Therefore, we derive :

\begin{equation}
  \label{eq:sum_pk2}
  \sum_{k=0}^{n-1}  \ps{a_{k+1},p_{k+1}^2} \leq \frac{H_0- \inf f}{\varepsilon}\,.
\end{equation}

Moreover, from the Algorithm~\ref{algo} second update rule, we get $\nabla f(x_k) = \frac{1}{b} p_{k+1} - \frac{1-b}{b} p_k$. Hence, we have for all $k \geq 0$ :
\begin{equation*}
  \nabla f(x_k)^2 \leq 2 \left(\frac{1}{b^2} p_{k+1}^2 + \frac{(1-b)^2}{b^2} p_k^2\right) \leq \frac{2}{b^2} (p_{k+1}^2 + p_k^2)\,.
\end{equation*}

We deduce that :
\begin{align*}
  \label{eq:sumgrad_sump}
  \sum_{k=0}^{n-1} \ps{a_{k+1},\nabla f(x_k)^2} &\leq \frac{2}{b^2} \sum_{k=0}^{n-1} \ps{a_{k+1},p_{k+1}^2 + p_k^2}\\
                    &= \frac{2}{b^2} \sum_{k=0}^{n-1} \ps{a_{k+1},p_{k+1}^2} + \frac{2}{b^2} \sum_{k=0}^{n-1} \ps{a_{k+1},p_k^2}\\
                    &\leq \frac{2}{b^2} \sum_{k=0}^{n-1} \ps{a_{k+1},p_{k+1}^2} + \frac{2}{b^2 \alpha} \sum_{k=0}^{n-1} \ps{a_k,p_k^2}\\
                    &\leq \frac{2}{b^2}(1+ \frac 1\alpha) \sum_{k=0}^{n} \ps{a_k,p_k^2}\\
                    &\leq \frac{2 (1+\alpha)}{b^2\alpha} \left(\frac{H_0 - \inf f}{\varepsilon}+ \ps{a_0, p_0^2}\right)  \,.
\end{align*}
\end{proof}

\subsection{Proof of \Cref{thm:deterministic}}
\label{proof:deterministic}

This is a consequence of \lemmaref{lemma:lyap}. Conditions $A_{n+1} \geq \varepsilon$ and $B \geq 0$ write as follow :
\begin{equation*}
  a_{n+1} \leq \frac{2}{L}\left( 1- \frac{b-(1-\alpha)}{2u} - \frac{1-\alpha}{2b} - \varepsilon \right) \quad \text{and} \quad u \leq \frac{\alpha b}{b-(1-\alpha)}\,.
\end{equation*}
We get the assumption made in the proposition by injecting the second condition into the first one and adding the assumption $\frac{a_{n+1}}{a_n} \leq \alpha$ made in the lemma.
Under this assumption, we sum over $0\leq k \leq n-1$ \Cref{eq:descent}, rearrange it and use $A_{n+1} \geq \varepsilon$, $B \geq 0$ and $a_{k+1} \geq \delta$ to obtain :

\begin{equation*}
\sum_{k=0}^{n-1}\delta \, \varepsilon \, \|p_{k+1}\|^2 \leq H_0 - H_n\,,
\end{equation*}
Then, observe that $H_n \geq f(x_n) \geq \inf f$. Therefore, we derive :

\begin{equation}
  \label{eq:sum_pk2}
  \sum_{k=0}^{n-1}  \|p_{k+1}\|^2 \leq \frac{H_0- \inf f}{\delta \varepsilon}\,.
\end{equation}

Moreover, from the algorithm~\ref{algo} second update rule, we get $\nabla f(x_k) = \frac{1}{b} p_{k+1} - \frac{1-b}{b} p_k$. Hence, we have for all $k \geq 0$ :
\begin{equation*}
  \|\nabla f(x_k)\|^2 \leq 2 \left(\frac{1}{b^2} \|p_{k+1}\|^2 + \frac{(1-b)^2}{b^2} \|p_k\|^2\right) \leq \frac{2}{b^2} (\|p_{k+1}\|^2 + \|p_k\|^2)\,.
\end{equation*}

We deduce that :
\begin{equation}
  \label{eq:sumgrad_sump}
  \sum_{k=0}^{n-1} \|\nabla f(x_k)\|^2 \leq \frac{2}{b^2} \sum_{k=0}^{n-1} (\|p_{k+1}\|^2 + \|p_k\|^2) = \frac{2}{b^2} \left(2 \sum_{k=1}^{n-1} \|p_k\|^2 +  \|p_n\|^2 + \|p_0\|^2 \right) \leq \frac{4}{b^2} \sum_{k=0}^{n} \|p_k\|^2  \,.
\end{equation}

Finally, using \Cref{eq:sum_pk2,eq:sumgrad_sump}, we have :
\begin{equation*}
  \min_{0\leq k \leq n-1} \|\nabla f(x_k)\|^2 \leq \frac{1}{n} \sum_{k=0}^{n-1} \|\nabla f(x_k)\|^2 \leq \frac{4}{nb^2} \left(\frac{H_0 - \inf f}{\delta \varepsilon}+\|p_0\|^2\right)\,.
\end{equation*}

\subsection{Proof of \Cref{thm:stochastic}}
\label{proof:stochastic}
The proof of this proposition mainly follows the same path as its deterministic counterpart.
However, due to stochasticity, a residual term (the last term in \Cref{eq:lyap_stochastic})
quantifying the difference between the stochastic gradient estimate and the true
gradient of the objective function (compare \Cref{eq:lyap_stochastic} to \lemmaref{lemma:lyap}) remains.
Following the exact same steps of \Cref{proof:lemma_descent}, we obtain by replacing
the deterministic gradient $\nabla f(x_n)$ by its stochastic estimate $\nabla f(x_n, \xi_{n+1})$ :

\begin{align}
H_{n+1} &\leq H_n - \ps{a_{n+1}p_{n+1}^2, 1-\frac{a_{n+1}L}{2} - \frac{|b-(1-\alpha)|}{2u} - \frac{1-\alpha}{2b}}\nonumber\\
        & - \frac{b}{2} \ps{a_{n+1}(\nabla f(x_n, \xi_{n+1})-p_n)^2, \left( 1-  \frac{|b-(1-\alpha)|u}{b} - (1-\alpha)\right) \mathbf{1}}\nonumber\\
        & + \ps{\nabla f(x_n, \xi_{n+1})-\nabla F(x_n),a_{n+1} p_{n+1}}\label{eq:lyap_stochastic}\,.
\end{align}

Using the classical inequality $xy \leq \frac{x^2}{2\eta} + \frac{\eta y^2}{2}$ with $\eta = 1/2$ and the almost sure boundedness of the step size $a_{n+1}$, we get :

\begin{align*}
\ps{\nabla f(x_n, \xi_{n+1})-\nabla F(x_n),a_{n+1} p_{n+1}} &\leq \ps{(\nabla f(x_n, \xi_{n+1})-\nabla F(x_n))^2+\frac14 p_{n+1}^2,a_{n+1}}\\
                                                            &\leq \bar a_{\sup} \|\nabla f(x_n, \xi_{n+1})-\nabla F(x_n)\|^2 + \frac 14 \ps{a_{n+1},p_{n+1}^2}\,.
\end{align*}

Therefore, taking the expectation and using the boundedness of the variance, we obtain from \Cref{eq:lyap_stochastic} :

\begin{equation*}
\bE[H_{n+1}] - \bE[H_n] \leq - \bE\left[\ps{a_{n+1}p_{n+1}^2, \frac 34 -\frac{a_{n+1}L}{2} - \frac{|b-(1-\alpha)|}{2u} - \frac{1-\alpha}{2b}} \right] + \bar a_{\sup}\sigma^2\,.
\end{equation*}

Then, the proof follows the lines of \Cref{proof:prop_cv-a_n}. Hence, we have
\begin{equation*}
\bE[H_{n+1}] - \bE[H_n] \leq - \bE\left[\ps{a_{n+1}p_{n+1}^2, \varepsilon \bf{1}} \right] + \bar a_{\sup}\sigma^2\,.
\end{equation*}

We sum these inequalities for $k = 0, \cdots, n-1$, inject the assumption $a_{n+1} \geq \delta$ and rearrange the terms to obtain
\begin{equation}
  \label{eq:bound_sto}
  \delta \,\bE\left[\sum_{k=0}^{n-1}  \|p_{k+1}\|^2 \right] \leq  \bE\left[\sum_{k=0}^{n-1}  \ps{a_{k+1},p_{k+1}^2}\right] \leq \frac{H_0 - \inf f}{\varepsilon} + \frac{n \bar a_{\sup} \sigma^2}{\varepsilon}\,.
\end{equation}

Then, using $\nabla f(x_k,\xi_{k+1}) = \frac{1}{b} p_{k+1} - \frac{1-b}{b} p_k$ and a similar upperbound to \Cref{eq:sumgrad_sump} we show that

\begin{equation}
  \label{eq:boundgrad_boundp}
  \sum_{k=0}^{n-1} \|\nabla f(x_k,\xi_{k+1})\|^2 \leq \frac{4}{b^2}\sum_{k=0}^{n}\|p_k\|^2\,.
\end{equation}

Therefore, combining \Cref{eq:boundgrad_boundp,eq:bound_sto}, we establish the following inequality

\begin{equation*}
 \bE\left[\sum_{k=0}^{n-1}\|\nabla f(x_k,\xi_{k+1})\|^2\right] \leq \frac{4}{b^2} \left(\frac{H_0 - \inf f}{\delta \varepsilon} + \|p_0\|^2 \right) + \frac{4 \bar a_{\sup} n}{\delta\varepsilon b^2} \sigma^2\,.
\end{equation*}

Finally, we apply Jensen's inequality to $\|\cdot\|^2$ and divide the previous inequality by $n$ to obtain the sought result

\begin{equation*}
\frac1n \sum_{k=0}^{n-1} \bE\left[\|\nabla F(x_k)\|^2\right] \leq \frac{4}{n\delta b^2} \left(\frac{H_0 - \inf f}{\delta \varepsilon} + \|p_0\|^2 \right) + \frac{4 \bar a_{\sup}}{\delta\varepsilon b^2} \sigma^2\,.
\end{equation*}

\begin{remark}
  \label{remark:without_lower_bound}
  Following the derivations in \Cref{proof:prop_cv-a_n}, note that we also obtain the following result

  \begin{equation*}
  \bE\left[\sum_{k=0}^{n-1}  \ps{a_{k+1}, \nabla f(x_k,\xi_{k+1})^2}\right] \leq \frac{2(1+\alpha)}{b^2 \alpha} \left(\frac{H_0 - \inf f}{\varepsilon} + \ps{a_0,p_0^2} + \frac{n \bar a_{\sup} \sigma^2}{\varepsilon}\right)\,.
  \end{equation*}
\end{remark}

\subsection{Comparison to \citet{ipiano}}
\label{appendix:comp_ipiano}

We recall the conditions satisfied by $\alpha_n$ and $\beta_n$ in \citet{ipiano} in order to traduce them in terms of the algorithm (\ref{algo}) at stake. Define :
\begin{equation*}
  \delta_n \eqdef \frac{1}{\alpha_n} - \frac{L}{2} - \frac{\beta_n}{2\alpha_n} \qquad  \gamma_n \eqdef \delta_n - \frac{\beta_n}{2\alpha_n}.
\end{equation*}
Conditions of \citet{ipiano} write: $\alpha_n \geq c_1$\, $\beta_n \geq 0$\, $\delta_n \geq \gamma_n \geq c_2$ where $c_1, c_2$ are positive constants and $(\delta_n)$ is monotonically decreasing.

One can remark that algorithm~(\ref{algo}) can be written as~(\ref{hb}) with step sizes $\alpha_n = b a_{n+1}$ and inertial parameters $\beta_n = (1-b)\frac{a_{n+1}}{a_n}$. Conditions on these parameters can be expressed in terms of $a_n$. Supposing $c_2 = 0$, the condition $\gamma_n \geq c_2$ is equivalent to

\begin{equation}
   \frac{a_{n+1}}{a_n} \leq \frac{2}{2-b(2-a_n L)}.
\end{equation}
 Note that the classical condition $a_n \leq 2/L$ shows up consequently. Moreover, the condition on $(\delta_n)$ is equivalent to

 \begin{equation}
   \label{condition:lowerbound}
  \frac{1}{a_{n+1}} \leq \frac{3-b}{2} \frac{1}{a_n} -\frac{1-b}{2a_{n-1}} \qquad \text{for} \qquad n \geq 1.
\end{equation}

Note that we get rid of condition~(\ref{condition:lowerbound}) while allowing adaptive step sizes $a_n$ (see Proposition~\ref{prop:cv-a_n}).

\subsection{Performance of gradient descent in the nonconvex setting.}
\label{proof:grad_descent}

In the nonconvex setting, for a smooth function $f$, we cannot say anything about the convergence rate
of the sequences $(f(x_k))$ and $(x_k)$. Nevertheless, as exposed in \citep[p.28]{nesterov2004book},
we can control the minimum of the gradients norms. We prove this result in the following for completeness.

Consider the gradient descent algorithm defined by : $x_{k+1} = x_k - \gamma \nabla f(x_k)$.
Assume that $\gamma >0$ and $1-\frac{\gamma L}{2} >0$.

Supposing that $\nabla f$ is $L-$Lipschitz, using Taylor's expansion and regrouping the terms, we obtain the following inequality:
\begin{equation*}
f(x_{k+1}) \leq f(x_k) - \gamma \left(1-\frac{\gamma L}{2}\right) \|\nabla f(x_k)\|_2^2.
\label{eq:fdiff}
\end{equation*}

Then, we sum the inequalities for $0 \leq k \leq n-1$, lower bound the gradients
norms in the sum by their minimum and we obtain for $n  \geq 1$ :

\begin{equation*}
 \min_{0\leq k \leq n-1} \|\nabla f(x_k)\|_2^2 \leq \frac{f(x_0) - \inf f}{n \gamma (1- \frac{\gamma L}{2})}.
\end{equation*}

\section{K\L{} Convergence Analysis}
\subsection{Three abstract conditions}
Inspired from the abstract convergence mechanism of \citet[Appendix]{bolte2018first},
we show that similar conditions hold in our case.
We highlight that these conditions are slightly different here, since we do not deal with
\textit{gradient-like descent sequences} (for which the objective function is nonincreasing over the iterations).
Conditions below are closer to those of \citet{ipiano} which studies a non-descent algorithm.
Note however that the Lyapunov function $H$ and the sequence $(z_k)$ we consider are different.

\begin{lemma}
\label{lemma:conditions}
   Let $(z_k)_{k \in \bN}$ be the sequence defined for all $k \in \bN$ by $z_k = (x_k,y_k)$ where $y_k = \sqrt{a_k}p_k$ and $(x_k,p_k)$ is generated by Algorithm~(\ref{algo}) from a starting point $z_0$. Let \Cref{hyp:model,hyp:stepsize} hold true. Assume moreover that condition~(\ref{hyp:stepsize_bound}) holds. Then,

  \begin{enumerate}[(i),noitemsep,leftmargin=*]

    \item (sufficient decrease property) There exists a positive scalar $\rho_1$ s.t. :
          \begin{equation*}
              H(z_{k+1}) - H(z_k) \leq - \rho_1 \, \|x_{k+1} -x_k\|^2 \quad \forall k \in \bN.
          \end{equation*}

    \item There exists a positive scalar $\rho_2$ s.t. :
          \begin{equation*}
              \|\nabla H(z_{k+1})\| \leq \rho_2 \,(\|x_{k+1} -x_k\| + \|x_{k} -x_{k-1}\|) \quad \forall k \geq 1.
          \end{equation*}

    \item (continuity condition) If $\bar z$ is a limit point of a subsequence $(z_{k_j})_{j \in \bN}$, then
     $\lim\limits_{j \rightarrow +\infty} H(z_{k_j}) = H(\bar z)$.
  \end{enumerate}
\end{lemma}

\begin{remark}
  Note that the conditions in \lemmaref{lemma:conditions} can be generalized to a nonsmooth objective function.
  Indeed, in \citet[Appendix]{bolte2018first}, the Fréchet subdifferential replaces the gradient.
\end{remark}

\begin{proof}
\begin{enumerate}[(i),leftmargin=*]

      \item From \Cref{lemma:lyap,thm:deterministic}, we get for all $k \in \bN$:
      \begin{equation*}
        H(z_{k+1}) - H(z_k) \leq - \varepsilon \ps{a_{k+1},p_{k+1}^2}
                            \leq - \varepsilon \ps{a_{k+1},\left(\frac{x_{k+1}-x_k}{-a_{k+1}} \right)^2}
                            \leq - \frac{\varepsilon}{a_{\sup}}\,\|x_{k+1} -x_k\|^2.
      \end{equation*}

      We set $\rho_1 \eqdef \frac{\varepsilon}{a_{\sup}}$.

      \item First, observe that for all $k \in \bN$
      \begin{equation}
        \label{eq:decomp}
          \|\nabla H(z_{k+1})\| \leq \|\nabla f(x_{k+1})\| + \frac{1}{b} \,\|y_{k+1}\|\,.
        \end{equation}

Now, let us upperbound each one of these two terms.
Recall that we can rewrite our algorithm under a "Heavy-ball"-like form as follows:
\begin{equation*}
  x_{k+1} = x_k - \alpha_k \nabla f(x_k) + \beta_k (x_k - x_{k-1}) \quad \forall k \geq 1.
\end{equation*}
where $\alpha_k \eqdef b a_{k+1}$ and $\beta_k = (1-b) \frac{a_{k+1}}{a_k}$ are vectors.

On the one hand, using the L-Lipschitz continuity of the gradient, we obtain
\begin{align*}
 \|\nabla f(x_{k+1})\|^2 & \leq 2 \left(\|\nabla f(x_{k+1}) - \nabla f(x_k)\|^2 +  \|\nabla f(x_k)\|^2\right)\\
                         & \leq 2 \left(L^2\,\| x_{k+1} - x_k \|^2 +  \|\nabla f(x_k)\|^2\right)
\end{align*}

Moreover,
\begin{align*}
\|\nabla f(x_k)\|^2 & = \left\|\frac{x_k -x_{k+1}}{\alpha_k} + \frac{\beta_k}{\alpha_k} (x_k - x_{k-1}) \right\|^2\\
                    & \leq 2 \left\|\frac{x_k -x_{k+1}}{ba_{k+1}}\right\|^2 + 2 \left\|\frac{1-b}{b}\frac{1}{a_k} (x_k - x_{k-1}) \right\|^2\\
                    & \leq \frac{2}{b^2 \delta^2}\,\| x_{k+1} - x_k \|^2 + \frac{2(1-b)^2}{b^2 \delta^2}\,\| x_{k} - x_{k-1} \|^2\\
                    & \leq \frac{2}{b^2 \delta^2}\,(\| x_{k+1} - x_k \|^2 + \| x_{k} - x_{k-1} \|^2).
\end{align*}

Hence,
\begin{align*}
\|\nabla f(x_{k+1})\|^2 & \leq 2 \left(L^2\,\| x_{k+1} - x_k \|^2 +  \|\nabla f(x_k)\|^2\right)\\
                        & \leq 2 \left( L^2 + \frac{2}{b^2 \delta^2} \right) \,\| x_{k+1} - x_k \|^2 + \frac{4}{b^2 \delta^2}\, \| x_{k} - x_{k-1} \|^2\\
                        &\leq 2 \left( L^2 + \frac{2}{b^2 \delta^2} \right) (\| x_{k+1} - x_k \|^2 + \| x_{k} - x_{k-1} \|^2)\,.
\end{align*}

Therefore, the following inequality holds :
\begin{equation*}
\|\nabla f(x_{k+1})\| \leq \sqrt{2 \left( L^2 + \frac{2}{b^2 \delta^2} \right)} (\| x_{k+1} - x_k \| + \| x_{k} - x_{k-1} \|)\,.
\end{equation*}

On the otherhand,
\begin{equation*}
  \|y_{k+1}\| = \|\sqrt{a_{k+1}}p_{k+1}\| = \left\|\frac{x_{k+1}-x_k}{\sqrt{a_{k+1}}}\right\| \leq \frac{1}{\sqrt{\delta}}\,\|x_{k+1}-x_k\|\,.
\end{equation*}

Finally, combining the inequalities for both terms in \Cref{eq:decomp}, we obtain

\begin{equation*}
\|\nabla H(z_{k+1})\| \leq \rho_2 (\| x_{k+1} - x_k \| + \| x_{k} - x_{k-1} \|) \quad \forall k \geq 1\,.
\end{equation*}
with $\rho_2 \eqdef  \left( \sqrt{2 \left( L^2 + \frac{2}{b^2 \delta^2} \right)} + \frac{1}{b\sqrt{\delta}} \right)$.

      \item This is a consequence of the continuity of $H$.
\end{enumerate}
\end{proof}

\subsection{Proof of \lemmaref{lemma:subseq}}

\begin{enumerate}[(i),leftmargin=*]
      \item By \Cref{thm:deterministic}, the sequence $(H(z_n))_{n \in \bN}$ is nonincreasing.
      Therefore, for all $n \in \bN$, $H(z_n) \leq H(z_0)$\, and hence $z_n \in \{z \, \text{:}\, H(z) \leq H(z_0)\}$\,.
      Since $f$ is coercive, $H$ is also coercive and its level sets are bounded.
      As a consequence, $(z_n)_{n \in \bN}$ is bounded and there exist $z_{*} \in \bR^d$ and a subsequence
      $(z_{k_j})_{j \in \bN}$ s.t. $z_{k_j} \to z_{*}$ as $j \to \infty$. Hence, $\omega(z_0) \neq \emptyset$\,.
      Furthermore, $\omega(z_0) = \bigcap_{q \in \bN} \overline{\bigcup_{k \geq q} \{z_k\}}$ is compact as an intersection
      of compact sets.
      \item First, $\text{crit} H = \text{crit} f \times \{0\}$ because $\nabla H(z) = (\nabla f(x), y/b)^T$.
      Let $z_{*} \in \omega(z_0)$. Recall that $x_{k+1}-x_k \rightarrow 0$ as $k \to \infty$ by \Cref{thm:deterministic}.
      We deduce from the second assertion of \lemmaref{lemma:conditions} that  $\nabla H(z_k) \rightarrow 0$ as $k \to \infty$\,.
      As $z_{*} \in \omega(z_0)$, there exists a subsequence $(z_{k_j})_{j \in \bN}$ converging to $z_{*}$. Then, by Lipschitz continuity
      of  $\nabla H$, we get that $\nabla H(z_{k_j}) \rightarrow \nabla H(z_{*})$ as $j \to \infty$\,. Finally, $\nabla H(z_{*}) = 0$
      since $\nabla H(z_k) \rightarrow 0$ and $(\nabla H(z_{k_j}))_{j \in \bN}$ is a subsequence of $(\nabla H(z_n))_{n \in \bN}$ \,.

      \item This point stems from the definition of limit points.
      Every subsequence of the sequence $(\sd(z_k, \omega(z_0)))_{k \in \bN}$  converges to zero
      as a consequence of the definition of $\omega(z_0)$.

      \item The sequence $(H(z_n))_{n \in \bN}$ is nonincreasing by \Cref{thm:deterministic}.
      It is also bounded from below because $H(z_k) \geq f(x_k) \geq \inf f$ for all $k \in \bN$.
      Hence we can denote by $l$ its limit. Let $\bar z \in \omega(z_0)$.
      There there exists a subsequence $(z_{k_j})_{j \in \bN}$ converging to $\bar z$ as $j \to \infty$\,.
      By the third assertion of \lemmaref{lemma:conditions}, $\lim\limits_{j \rightarrow +\infty} H(z_{k_j}) = H(\bar z)$\,.
      Hence this limit equals $l$ since $(H(z_n))_{n \in \bN}$ converges towards $l$. Therefore, the restriction
      of $H$ to $\omega(z_0)$ equals~$l$\,.
\end{enumerate}

\subsection{Proof of \Cref{thm:rates}}
\label{proof:KL_rates}

The first step of this proof follows the same path as \citet[Proof of Theorem 6.2, Appendix]{bolte2018first}.
Since $f$ is coercive, $H$ is also coercive. The sequence $(H(z_k))_{k \in \bN}$ is nonincreasing. Hence, $(z_k)$ is bounded and
there exists a subsequence $(z_{k_q})_{q \in \bN}$ and $\bar z \in \bR^{2d}$ s.t. $z_{k_q} \to \bar z$ as $q \to \infty$\,. Then, since
$(H(z_k))_{k \in \bN}$ is nonincreasing and lowerbounded by $\inf f$, it is convergent and we obtain by continuity of $H$,
\begin{equation}
  \label{eq:lim_H}
\lim\limits_{k \rightarrow +\infty} H(z_k) = H(\bar z)\,.
\end{equation}
Using \Cref{thm:deterministic}, observe that the sequence $(y_k)$ converges to zero since $(a_k)$ is bounded and $p_k \to 0$.
If there exists $\bar k \in \bN$\, s.t. $H(z_{\bar k}) = H(\bar z)$\,, then $H(z_{\bar k +1}) = H(\bar z)$ and by
the first point of \lemmaref{lemma:conditions}, $x_{\bar k +1} = x_{\bar k}$ and then $(x_k)_{k \in \bN}$ is stationary
and for all $k \geq \bar k$\,, $H(z_k) = H(\bar z)$ and the results of the theorem hold in this case (note that $\bar z \in \text{crit} H$ by \lemmaref{lemma:subseq}). Therefore, we can assume now that $H(\bar z) < H(z_k) \forall k >0$\, since $(H(z_k))_{k \in \bN}$ is nonincreasing and
\Cref{eq:lim_H} holds. One more time, from \Cref{eq:lim_H}, we have that for all $\eta > 0$, there exists $k_0 \in \bN$ s.t. $H(z_k) < H(\bar z) + \eta$\, for all
$k > k_0$. From \lemmaref{lemma:subseq}, we get $\sd(z_k, \omega(z_0)) \to 0$ as $k \rightarrow +\infty$\,. Hence, for all $\varepsilon > 0$, there exists $k_1 \in \bN$ s.t. $\sd(z_k, \omega(z_0)) < \varepsilon$ for all $k > k_1$\,. Moreover, $\omega(z_0)$ is a nonempty compact set and $H$ is finite and constant on it. Therefore, we can apply the uniformization \lemmaref{lemma:unif_KL} with $\Omega = \omega(z_0)$. Hence, for any $k > l \eqdef \max(k_0,k_1)$, we get
\begin{equation}
  \label{eq:KL_H}
\varphi'(H(z_k)-H(\bar z))^2 \, \|\nabla H(z_k)\|^2 \geq 1\,.
\end{equation}
This completes the first step of the proof.
In the second step, we follow the proof of \citet[Theorem 2]{johnstone2017convergence}.
Using \lemmaref{lemma:conditions}~.(i)-(ii), we can write for all $k \geq 1$,
\begin{equation*}
\|\nabla H(z_{k+1})\|^2 \leq 2 \rho_2^2 \,(\|x_{k+1} -x_k\|^2 + \|x_{k} -x_{k-1}\|^2) \leq \frac{2 \rho_2^2}{\rho_1} (H(z_{k-1})-H(z_{k+1}))\,.
\end{equation*}

Injecting the last inequality in \Cref{eq:KL_H}, we obtain for all $k > k_2 \eqdef \max(l,2)$,
\begin{equation*}
\frac{2 \rho_2^2}{\rho_1}\,\varphi'(H(z_k)-H(\bar z))^2\, (H(z_{k-2})-H(z_{k})) \geq 1\,.
\end{equation*}

Now, use $\varphi'(s) = \bar c s^{\theta -1}$ to derive the following for all $k > k_2$:
\begin{equation}
  \label{eq:key}
[H(z_{k-2})-H(\bar z)] - [H(z_k)-H(\bar z)] \geq \frac{\rho_1}{2 \rho_2^2 \, \bar{c}^2} [H(z_k)-H(\bar z)]^{2(1-\theta)}\,.
\end{equation}

Let $r_k \eqdef  H(z_k) - H(\bar z)$ and $C_1 = \frac{\rho_1}{2 \rho_2^2 \, \bar{c}^2}$. Then, we can rewrite \Cref{eq:key} as
\begin{equation}
  \label{eq:rec_rate}
r_{k-2} - r_k \geq C_1 r_k^{2(1-\theta)} \quad \forall k > k_2\,.
\end{equation}

We distinguish three different cases to obtain the sought results.
\begin{enumerate}[(i),leftmargin=*]
\item \underline{$\theta = 1$}\,:\\
Suppose $r_k > 0$ for all $k >k_2$\,. Then, since we know that $r_k \to 0$ by \Cref{eq:lim_H}, $C_1$ must be
equal to $0$. This is a contradiction. Therefore, there exist $k_3 \in \bN$ s.t. $r_k = 0$ for all $k > k_3$\, (recall that
$(r_k)_{k \in \bN}$ is nonincreasing).

\item \underline{$\theta \geq \frac{1}{2}$}\,:\\
As $r_k \to 0$, there exists $k_4 \in \bN$ s.t. for all $k \geq k_4,$\, $r_k \leq 1$\,. Observe that $2(1-\theta) \leq 1$ and hence
$r_{k-2} - r_k \geq C_1 r_k$ for all $k > k_2$\, and then
\begin{equation}
  \label{eq:lin_rate}
r_k \leq (1+C_1)^{-1} r_{k-2} \leq  (1+C_1)^{-p_1} r_{k_4}\,.
\end{equation}
where $p_1 \eqdef \lfloor \frac{k-k_4}{2} \rfloor$\,. Notice that $p_1 > \frac{k-k_4-2}{2}$. Thus, the linear convergence
result follows.
Note also that if $\theta = 1/2$,\, $2(1- \theta) = 1$ and \Cref{eq:lin_rate} holds for all $k>k_2$\,.

\item \underline{$\theta < \frac{1}{2}$}\,:\\
Define the function $h$ by $h(t) = \frac{D}{1-2\theta} t^{2\theta -1}$ where $D > 0$ is a constant. Then,

\begin{equation*}
h(r_k) - h(r_{k-2}) = \int_{r_{k-2}}^{r_k} h'(t)dt = D \int_{r_k}^{r_{k-2}} t^{2\theta-2} dt \geq D\,(r_{k-2}-r_k)\,r_{k-2}^{2\theta -2}\,.
\end{equation*}

We disentangle now two cases :
\begin{enumerate}[(a)]
\item Suppose $2 r_{k-2}^{2\theta-2} \geq r_k^{2\theta-2}$. Then, by \Cref{eq:rec_rate}, we get
\begin{equation}
  \label{eq:case1}
h(r_k) - h(r_{k-2}) = D\,(r_{k-2}-r_k)\,r_{k-2}^{2\theta -2} \geq \frac{C_1 \,D}{2}\,.
\end{equation}

\item Suppose now the opposite inequation $2 r_{k-2}^{2\theta-2} < r_k^{2\theta-2}$.
We can suppose without loss of generality that $r_k$ are all positive. Otherwise,
if there exists $p$ such that $r_p = 0$, the sequence
$(r_k)_{k \in \bN}$ will be stationary at $0$ for all $k \geq p$\,.
Observe that $2\theta -2 < 2\theta -1 < 0$, thus $\frac{2 \theta -1}{2 \theta -2} > 0$\,.
As a consequence, we can write in this case $r_k^{2\theta -1} > q\, r_{k-2}^{2\theta-1}\,$
where $q \eqdef 2^{\frac{2 \theta -1}{2 \theta -2}} > 1$\,. Therefore, using moreover that
the sequence $(r_k)_{k \in \bN}$ is nonincreasing and $2 \theta -1 < 0$, we derive the following

\begin{equation}
\label{eq:case2}
h(r_k) - h(r_{k-2}) = \frac{D}{1-2\theta}(r_k^{2\theta-1}-r_{k-2}^{2\theta-1}) >  \frac{D}{1-2\theta}\,(q-1)r_{k-2}^{2\theta-1} > \frac{D}{1-2\theta}\,(q-1)r_{k_2}^{2\theta-1} \eqdef C_2\,.
\end{equation}
\end{enumerate}

Combining \Cref{eq:case1} and \Cref{eq:case2} yields $h(r_k) \geq  h(r_{k-2}) + C_3$ where $C_3 \eqdef \min(C_2,\frac{C_1 \,D}{2})$\,.
Consequently, $h(r_k) \geq h(r_{k-2\,p_2}) + p_2\,C_3$\, where $p_2 \eqdef \lfloor \frac{k-k_2}{2} \rfloor$\,. We deduce from this inequality that
\begin{equation*}
h(r_k) \geq h(r_k) - h(r_{k-2\,p_2}) \geq p_2\,C_3\,.
\end{equation*}

Therefore, rearranging this inequality using the definition of $h$, we obtain $r_k^{1-2\theta} \leq \frac{D}{1-2\theta} (C_3\,p_2)^{-1}$\,. Then, since $p_2 > \frac{k-k_2-2}{2}$\,,

\begin{equation*}
r_k \leq C_4\,p_2^{\frac{1}{2\theta-1}} \leq C_4 \left(\frac{k-k_2-2}{2}\right)^{\frac{1}{2\theta-1}}\,.
\end{equation*}
where $C_4 \eqdef \left(\frac{C_3\,(1-2\theta)}{D}\right)^{\frac{1}{2\theta-1}}$\,.
\end{enumerate}
We conclude the proof by observing that $f(x_{k}) \leq H(z_k)$ and recalling that $\bar z \in \text{crit} H$\,.

\subsection{Proof of \lemmaref{lemma:f_KL}}

Since $f$ has the K\L{} property at $\bar x$ with an exponent $\theta \in (0,1/2]$, there exist
$c, \varepsilon$ and $\nu > 0$ s.t.
\begin{equation}
  \label{eq:proof_kl}
  \|\nabla f(x)\|^{\frac{1}{1-\theta}} \geq c (f(x)-f(\bar x))
  \end{equation}

  for all $x \in \bR^d$ s.t. $\|x-\bar x\|\leq \varepsilon$\, and $f(x) < f(\bar x) + \nu$
  where condition $f(\bar x)-f(x)$ is dropped because \Cref{eq:proof_kl} holds trivially otherwise.
  Let $z = (x,y) \in \bR^{2d}$ be s.t. $\|x-\bar x\|\leq \varepsilon$\,, $\|y\| \leq \varepsilon$
  and $H(\bar x,0) < H(x,y) < H(\bar x,0) + \nu$\,. We assume that $\varepsilon < b$\, ($\varepsilon$ can
  be shrunk if needed). We have $f(x) \leq  H(x,y) < H(\bar x,0) + \nu = f(\bar x) + \nu$\,. Hence \Cref{eq:proof_kl} holds for these $x$.

  By concavity of $u \mapsto u^{\frac{1}{2(1-\theta)}}$\,, we obtain
  $$
  \|\nabla H(x,y)\|^{\frac{1}{1-\theta}} \geq C_0\,\left( \|\nabla f(x)\|^{\frac{1}{1-\theta}} + \left\|\frac{y}{b}\right\|^{\frac{1}{1-\theta}}  \right)
  $$
  where $C_0 \eqdef 2^{\frac{1}{2(1-\theta)}-1}$\,.

Hence, using \Cref{eq:proof_kl}, we get
$$
\|\nabla H(x,y)\|^{\frac{1}{1-\theta}} \geq C_0\,\left( c\,(f(x)-f(\bar x)) + \left\|\frac{y}{b}\right\|^{\frac{1}{1-\theta}}  \right)\,.
$$

Observe now that $\frac{1}{1-\theta} \geq 2$ and  $\left\|\frac{y}{b}\right\| \leq \frac{\varepsilon}{b} \leq 1$. Therefore, $\left\|\frac{y}{b}\right\|^{\frac{1}{1-\theta}} \geq  \|y/b\|^2$\,.

Finally,
\begin{align*}
\|\nabla H(x,y)\|^{\frac{1}{1-\theta}} &\geq C_0\,\left( c\,(f(x)-f(\bar x)) +  \frac{2}{b}\frac{1}{2b}\|y\|^2 \right) \\
                                       &\geq C_0\,\min\left(c,\frac{2}{b}\right)\,\left(f(x)-f(\bar x) + \frac{1}{2b}\|y\|^2 \right)\\
                                       & = C_0\,\min\left(c,\frac{2}{b}\right)\,\left( H(x,y)-H(\bar x,0) \right)\,.
\end{align*}

This completes the proof.
\setcounter{page}{240}
\pagestyle{empty}

%
%
%

\end{document}